\documentclass{article}

\usepackage{arxiv}

\usepackage[utf8]{inputenc} % allow utf-8 input
\usepackage[T1]{fontenc}    % use 8-bit T1 fonts
\usepackage{hyperref}       % hyperlinks
\usepackage{url}            % simple URL typesetting
\usepackage{booktabs}       % professional-quality tables
\usepackage{amsfonts}       % blackboard math symbols
\usepackage{nicefrac}       % compact symbols for 1/2, etc.
\usepackage{microtype}      % microtypography
\usepackage{graphicx}
\usepackage{natbib}
\usepackage{doi}
\usepackage{physics,amsmath}
\usepackage{authblk}
\usepackage{amsthm}
\usepackage{placeins}
\usepackage{soul}
\usepackage{xcolor}

\newtheorem{theorem}{Theorem}[section]
\newtheorem{lemma}[theorem]{Lemma}

\newtheorem{remark}[theorem]{Remark}
\newcommand{\average}[1]{\ensuremath{\{\!\{#1\}\!\}} }
\newcommand{\jump}[1]{\ensuremath{[\![#1]\!]} }

\title{An entropy stable discontinuous Galerkin method for the spherical thermal shallow water equations}

\author[1]{Kieran Ricardo}
\author[1]{Kenneth Duru}
\author[2]{David Lee}
\affil[1]{Mathematical Sciences Institute, Australian National University, Canberra, Australia}
\affil[2]{Bureau of Meteorology, Melbourne, Australia}

\date{} 					% Or removing it

\begin{document}
\maketitle

\begin{abstract}
We present a novel discontinuous Galerkin finite element method for numerical simulations of  the rotating thermal shallow water equations in complex geometries using curvilinear meshes, with arbitrary accuracy. We derive an entropy functional which is convex, and  which must be preserved in order to preserve model stability at the discrete level. The numerical method is provably entropy stable and conserves mass, buoyancy, vorticity, and energy. This is achieved by using novel entropy stable numerical fluxes, summation-by-parts principle, and splitting the pressure and convection operators so that we can circumvent the use of chain rule at the discrete level. Numerical simulations on a cubed sphere mesh are presented to  verify the theoretical results. The numerical experiments demonstrate the robustness of the method for a regime of well developed turbulence, where it can be run stably without any dissipation. The entropy stable fluxes are sufficient to control the grid scale noise generated by geostrophic turbulence, eliminating the need for artificial stabilisation. 
\end{abstract}

% keywords can be removed
% \keywords{First keyword \and Second keyword \and More}

\section{Introduction}

The rotating thermal shallow water equations (TRSW) have recently gained attention as a test bed for atmospheric models. These equations extend the rotating shallow water equations (RSW) to include a temperature-like quantity, known as buoyancy, that is transported by the flow and modulates the hydrostatic pressure gradient forcing. As a result the TRSW possess an identical skew-symmetric Hamiltonian structure to the compressible Euler equations while still being a two dimensional system and more computationally efficient to simulate than the three dimensional compressible Euler equations.
Also, the structural similarity between TRSW and the Euler equations considerably simplifies the extension of  numerical models for TRSW to the compressible Euler equations, as opposed to the more traditional RSW. Therefore, the TRSW is an attractive equation set with which to prototype atmospheric models. However, the development of robust numerical methods for the TRSW on curvilinear meshes pose a significant challenge. 
% For example, a straightforward extension of the energy stable scheme in \cite{ricardo2023entropy} rapidly becomes numerical unstable. We propose that this is because, unlike the RSW, energy is not an entropy of the TRSW and therefore energy conservation alone is insufficient for numerical stability.
The desired properties for atmospheric models, as outlined in \cite{staniforth2012horizontal}, serve as the basis for TRSW-specific criteria detailed in \cite{eldred2019quasi}. These desirable properties include
\begin{enumerate}
    \item Conservation of total mass
    \item Conservation of total buoyancy
    \item Conservation of total energy
    \item Conservation of total vorticity
    \item Compatible advection of buoyancy, an initially constant buoyancy will remain constant
    \item %\st{Arbitrarily} 
    High order accurate
\end{enumerate}
Where we have omitted properties pertaining to the implied RSW solver for brevity. In addition to these properties we argue that a discrete method for the TRSW should also conserve buoyancy variance to ensure numerical stability.

In \cite{eldred2019quasi} a planar compatible mixed finite element method for the TRSW is developed that achieves the above properties. In parallel to this, \cite{kurganov2020well} and \cite{kurganov2021thermal} present a well-balanced finite volume for the TRSW. While both methods accurately simulate the planar test cases outlined in their respective articles, neither method has been extended to curvilinear meshes or the spherical geometry. 

As mentioned above, while there are finite element method \cite{eldred2019quasi} and finite volume method \cite{kurganov2020well,kurganov2021thermal} for TRSW, to the best of our knowledge, there is no DG method for the TRSW in literature. More importantly, we are not aware of any numerical method for TRSW applicable to  curvilinear meshes or the spherical geometry, which is necessary for many practical problems.  However, there is a large body of work on DG methods for the structurally similar compressible Euler equations \cite{waruszewski2022entropy, ranocha2020entropy, ducros2000high, morinishi2010skew, sjogreen2019entropy, gassner2016split, hennemann2021provably, pirozzoli2010generalized}. These methods utilize the conservative form of the Euler equations and aim to conserve the physical entropy by employing techniques such as splitting derivative operators, summation by parts (SBP), and entropy stable numerical fluxes. %In contrast, we employ these same techniques to ensure numerical stability through discrete mathematical entropy stability for the vector invariant TRSW.

In the current study, we focus on the the vector invariant form of the TRSW that is amenable to the design of mimetic numerical methods \cite{ricardo2023conservation}, and which can potentially preserve important invariants in the system. However, as opposed to the more traditional RSW \cite{ricardo2023conservation}, the energy functional is no longer a convex function of the prognostic variables and will not ensure numerical stability when preserved. First, we derive a new entropy functional for the TRSW, the buoyancy variance, and prove that it is conserved by the TRSW. The  proof of conservation of buoyancy variance  relies on the chain rule which is difficult to discretely mimic. To design an energy and entropy stable numerical method, we first derive a split form of the system that negates the need for the chain rule in the entropy conservation proof. This enables us to prove entropy and stability using only integration by parts.

Second, we derive a DG method for the split form of the TRSW on curvilinear meshes. For the discrete analysis we employ structure preserving curvilinear coordinate transformations and the summation-by-parts (SBP) principle to cancel the volume terms within the DG elements. The remaining surface terms can be eliminated using accurate numerical fluxes that connect locally adjacent elements.  To complete the DG method we construct a family of novel numerical fluxes that conserve/dissipate energy and entropy. We then prove semi-discrete entropy and energy stability by leveraging the SBP property of DG-SEM \cite{gassner2013skew} and our numerical fluxes.

Finally, we present detailed numerical simulations on a cubed sphere mesh and  verify the theoretical results. The numerical experiments demonstrate the robustness of the method for a regime of well developed turbulence, where it can be run stably without any artificial dissipation. The entropy stable fluxes are sufficient to control the grid scale noise generated by geostrophic turbulence, eliminating the need for artificial stabilisation.
Our approach has been successfully extended to compatible MFE methods for both the TRSW and Euler equations \cite{ricardo2023entropy}, and extension to a DG method for the Euler equations is an ongoing work.

The rest of the article is as follows. Section 2 presents the continuous TRSW and the conservation properties that our scheme discretely mimics. Section 3 presents the  DG method. In section 4, we present proofs of the conservation and stability properties of the numerical model. Section 5 presents detailed numerical experiments which verify the theoretical analysis performed in the paper. In section 6, we draw conclusions and suggests directions for future work.

\section{Continuous thermal shallow water equations}

In this section we introduce the TRSW on the surface of a spherical
geometry and derive the continuous properties that our numerical scheme
should discretely mimic. While much of this analysis exists in the literature \cite{eldred2019quasi} we include it here to aid the exposition, and the analysis performed in this paper.
%our discrete proofs.  

\subsection{Vector invariant equations}

The TRSW equations are an extension of the RSW equations, incorporating a variable fluid density that is advected by the flow. The fluid density is represented by the buoyancy term $b$, which is defined as the scaled fluid density $b=g\tfrac{\rho}{\bar{\rho}}$, where $\rho$ represents density and $\bar{\rho}$ is a constant reference density. We consider the vector invariant form of the TRSW on a 2D manifold $\Omega$ embedded in $R^3$ with periodic boundary conditions. We choose periodic boundary conditions as they simplify the analysis and the domain of most practical interest is the sphere. The vector invariant TRSW are  
\begin{equation}
	\vb{u}_t + \omega\vb{k}\cross\vb{u} + \nabla G +\tfrac{1}{2}b\nabla h= 0\text{,}
	\label{eq:momentum}
\end{equation}
\begin{equation}
	h_t + \nabla \cdot \vb{F} = 0\text{,}
	\label{eq:mass}
\end{equation}
\begin{equation}
	(hb)_t + \nabla \cdot \vb{B} = 0\text{,}
	\label{eq:buoyancy}
\end{equation}
with
\begin{equation}\label{eq:fluxes_and_vorticity}
	\omega = \vb{k} \cdot \nabla \cross \vb{u} + f\text{,} \quad \vb{B} = hb\vb{u}\text{,} \quad \vb{F} = h\vb{u}\text{,} \quad G = \tfrac{1}{2}\vb{u}\cdot\vb{u} + \tfrac{1}{2}ghb\text{,}
\end{equation}
where $t\ge 0$ denotes the time variable, the unknowns are the flow velocity vector $\vb{u}=[u, v, w]^T$, the fluid depth $h$, and the mass-weighted buoyancy $hb$. Note that here and in the rest of the paper all differential operators are taken to be constrained to the manifold. Furthermore, $\omega$ is the absolute vorticity, $f$ is the Coriolis frequency, $\vb{F}$ is the mass flux, $G$ is the potential, $\vb{k}$ is the outward unit normal of the manifold, and $b$ is the buoyancy. At $t=0$, we augment the TRSW \eqref{eq:momentum}--\eqref{eq:buoyancy} with the initial conditions
\begin{align}
    \vb{u} = \vb{u}_0(\mathbf{x}), \quad {h} = {h}_0(\mathbf{x}), \quad {b} = {b}_0(\mathbf{x}), \quad \mathbf{x} \in \Omega.
\end{align}
such that the flow is constrained to the manifold $\vb{u}_0 \cdot \vb{k} = 0$. The constraint of the differential operators and the initial constraint $\vb{u}_0 \cdot \vb{k} = 0$ is sufficient to ensure that the $\vb{u}(\vb{x}, t) \cdot \vb{k} = 0$ holds $\forall \vb{x}\in \Omega,\; t \geq 0$. Note that if ${b}_0(\mathbf{x})=g$ is initially constant, it remains constant and the TRSW reduce to the regular RSW.

% The buoyancy acts as a variable gravitational acceleration and is related to the fluid density by $b=g\tfrac{\rho}{\bar{\rho}}$, where $\rho$ is density and $\bar{\rho}$ is a constant reference density. For the case of constant buoyancy $b=g$, the TRSW reduce to the regular RSW. 

The TRSW differs from the RSW in two main aspects. Firstly, the hydrostatic pressure term is modified to account for the effects of variable density, becoming $\tfrac{1}{2}\nabla hb + \tfrac{1}{2}b\nabla h$. Secondly, the TRSW introduces the conservation equation for $hb$. Despite the functional similarities between TRSW and RSW, the inclusion of density variations presents unexpected numerical challenges. In contrast to the RSW, ensuring energy stability alone does not guarantee numerical stability for TRSW models. Our experience is that energy-bounded extensions of stable RSW methods to TRSW tend to exhibit numerical instability over relatively short time frames. We hypothesise that this discrepancy arises from the fact that while energy serves as an entropy for the RSW, it does not hold true for the TRSW.

% . Despite the functional
% similarity of the TRSW and RSW, the inclusion of density variations pose
% surprising numerical challenges. In contrast to the RSW, energy stability
% does not appear to be sufficient for numerical stability of TRSW models.
% Energy stable extensions of stable RSW methods to the TRSW tend to be-
% come numerical unstable on relatively short time frames. We hypothesize
% that the reason for this discrepancy is that while energy is an entropy of the
% RSW, this does not hold for the TRSW.

\subsection{Conservation of mass and mass-weighted buoyancy}

We present the result in the following theorem.
\begin{theorem}\label{theo:conservation_of_mass_bouyancy}
Consider the TRSW \eqref{eq:momentum}--\eqref{eq:buoyancy} on a 2D manifold $\Omega$ embedded in $R^3$ with periodic boundary conditions. At time $t \ge 0$, let the total mass be denoted by $M(t) = \int_\Omega{hd\Omega}$ and the total mass-weighted buoyancy  denoted by $S(t) = \int_\Omega{hbd\Omega}$. We have
\begin{align}
    \frac{d M}{dt} = 0, \quad \& \quad \frac{d S}{dt} = 0, \quad \forall t\ge 0.
\end{align}
%is conserved in a periodic domain.
\end{theorem}
\begin{proof}
Integrating the continuity equation \eqref{eq:mass} in space over the domain $\Omega$, gives
%The time derivative of the total mass on a periodic domain is
\begin{equation}
    \frac{d M}{dt} = \frac{d }{dt} \int_\Omega{h d\Omega} = -\int_\Omega{\nabla \cdot \vb{F}d\Omega} = -\int_{\partial\Omega}{\vb{F}\cdot \vb{n}dl} = 0\text{,}
\end{equation}
where $\vb{n}$ is the unit normal of the boundary. The final step results from cancellation across periodic boundaries.
\\
\\
Similarly, we integrate the buoyancy equation \eqref{eq:buoyancy} over the domain $\Omega$
\begin{equation}
    \frac{d S}{dt} = \frac{d }{dt}\int_\Omega{(hb) d\Omega} = -\int_\Omega{\nabla \cdot \vb{B}d\Omega} = -\int_{\partial\Omega}{\vb{B}\cdot \vb{n}dl} = 0\text{.}
\end{equation}
As before, because of the periodic boundary conditions, the boundary terms cancel out. The proof is complete.
\end{proof}

\subsection{Conservation of total vorticity}

We can derive an evolution equation for the vorticity $\omega$ by applying $\vb{k} \cdot \nabla \cross$ to the momentum equation \eqref{eq:momentum} giving

\begin{equation}\label{eq:vorticity_continuity}
	\omega_t + \nabla \cdot (\omega\vb{u}) + \tfrac{1}{2} \vb{k}  \cdot (\nabla b \cross\left(\nabla h\right))  = 0\text{.}
\end{equation}
The evolution equation \eqref{eq:vorticity_continuity} for local absolute vorticity $\omega$ satisfies a conservation law with a source term. Unlike the RSW \cite{ricardo2023conservation}, this implies, for the TRSW vorticity can be locally generated or dissipated. However, for constant buoyancy $b \equiv \emph{const.}$, the source term will vanish and we will regain the RSW \cite{ricardo2023conservation} which locally conserves absolute vorticity.

We will show below that for the TRSW \eqref{eq:momentum}--\eqref{eq:buoyancy} total absolute vorticity is globally conserved.
%\todo[inline]{It seems that I am misapplying the vector identities.}
\begin{theorem}\label{theo:conservation_of_vorticity}
Consider the TRSW \eqref{eq:momentum}--\eqref{eq:buoyancy} on a 2D manifold $\Omega$ embedded in $R^3$ with periodic boundary conditions. At time $t \ge 0$, let the total absolute vorticity be denoted by $\mathcal{W}(t) = \int_\Omega{\omega d\Omega}$, where $\omega$ is the absolute vorticity defined in \eqref{eq:fluxes_and_vorticity}. If the Coriolis frequency $f$ is time-invariant then we have
\begin{align}
    \frac{d \mathcal{W}(t)}{dt} = 0,  \quad \forall t\ge 0.
\end{align}
%is conserved in a periodic domain.
\end{theorem}
\begin{proof}
Defining $\omega = \vb{k} \cdot \nabla \cross\vb{u} + f$ and 
integrating it  in space, over the spatial  domain $\Omega$, gives
%The time derivative of the total mass on a periodic domain is
\begin{equation}
\begin{split}
         \mathcal{W} =  \int_\Omega{\omega d\Omega} = \int_\Omega{\left(\vb{k} \cdot \nabla \cross\vb{u} + f\right)d\Omega} = \int_{\partial\Omega}{\vb{k} \cdot\left(\vb{u}\cross \vb{n}\right)dl} + \int_\Omega{f d\Omega} 
\end{split}
\end{equation}
where $\vb{n}$ is the unit normal of the boundary ${\partial\Omega}$.
On the periodic boundaries, the boundary terms cancel out and we have
\begin{equation*}
\begin{split}
         \mathcal{W} =  \int_\Omega{f d\Omega}. 
\end{split}
\end{equation*}
Therefore if $f$ is time-invariant, then we have
\begin{equation*}
\begin{split}
         \frac{d \mathcal{W}}{dt} =  0, 
\end{split}
\end{equation*}
 which completes the proof. 
\end{proof}

\subsection{Conservation of energy}
We will now show that the energy in the medium is conserved. To begin, for the TRSW, we introduce the elemental energy $e$ and the total energy $\mathcal{E}(t)$ defined by
\begin{align}
    e = \tfrac{1}{2}h|\vb{u}|^2 + \tfrac{1}{2}h(hb), \quad \mathcal{E}(t) = \int_\Omega{e d\Omega}.
\end{align}
We can derive an evolution  equation for $e$ by taking the functional derivative of $e$ w.r.t. the prognostic variables 
\begin{equation}
    \dfrac{\partial e}{\partial \vb{u}} = \vb{F}\text{,} \quad\dfrac{\partial e}{\partial h} = G\text{,}\quad
    \dfrac{\partial e}{\partial hb} = \tfrac{1}{2}h\text{,}
\end{equation}
and multiplying by \eqref{eq:momentum}--\eqref{eq:buoyancy} respectively. Recalling the definitions of $\vb{F}=h\vb{u}$ and $\vb{B} = b\vb{F}$ we derive the following conservation law for the elemental energy $e$. Consider
\begin{equation}
\begin{split}
    e_t &= \vb{F} \cdot \vb{u}_t + Gh_t + \tfrac{1}{2}h(hb)_t \\
    &= -\vb{F} \cdot \nabla G -\tfrac{1}{2}b\vb{F}\cdot \nabla h -  G \nabla\cdot \vb{F} - \tfrac{1}{2}h \nabla \cdot \vb{B}
    \\
    &= -\nabla \cdot \left(\left(G + \tfrac{1}{2}hb\right)\vb{F} \right)
    \text{,}
\end{split}
\label{eq:energy-evol_0}
\end{equation}
giving
\begin{equation}
\begin{split}
    e_t + \nabla \cdot \left(\left(G + \tfrac{1}{2}hb\right)\vb{F} \right)= 0
    \text{,}
\end{split}
\label{eq:energy-evol}
\end{equation}
where we have used $\vb{F} \cdot \omega \vb{k} \cross \vb{u} = h\vb{u} \cdot \omega \vb{k} \cross \vb{u} = 0$. As before, global conservation directly follows from the local conservation law \eqref{eq:energy-evol}.
\begin{theorem}\label{theo:conservation_of_energy}
Consider the TRSW \eqref{eq:momentum}--\eqref{eq:buoyancy} on a 2D manifold $\Omega$ embedded in $R^3$ with periodic boundary conditions. At time $t \ge 0$, let the total energy be denoted by $\mathcal{E}(t) = \int_\Omega{ed\Omega}$. We have
\begin{align}
    \frac{d \mathcal{E}}{dt} = 0, \quad \forall t\ge 0.
\end{align}

\end{theorem}
\begin{proof}
Integrating the energy evolution equation \eqref{eq:energy-evol} in space, over the domain $\Omega$, gives
%The time derivative of the total mass on a periodic domain is
\begin{equation}
    \frac{d \mathcal{E}}{dt} = \frac{d }{dt} \int_\Omega{ed\Omega} = -\int_\Omega{\nabla \cdot \big(G+hb\big)\vb{F}d\Omega} = -\int_{\partial\Omega}{\big(G+hb\big)\vb{F}\cdot \vb{n}dl} = 0\text{,}
\end{equation}
where $\vb{n}$ is the unit normal of the boundary. The final step results from cancellation across periodic boundaries. The proof is complete.
\end{proof}

\subsection{Entropy}

Here we will show that elemental buoyancy variance, 
\begin{equation}\label{eq:buoyancyvariance}
\zeta=\tfrac{1}{2}hb^2 = \tfrac{1}{2}\dfrac{(hb)^2}{h},
\end{equation}
is an entropy function of the TRSW with our choice of prognostic variables, that is it is conserved and a convex function of the prognostic variables \cite{leveque1992numerical}. The importance of entropy functions is that they are conserved in smooth regions and are dissipated across discontinuities in physically relevant weak solutions. DG methods are continuous within each element but can be discontinuous across element boundaries. This motivates that the volume terms of DG methods should conserve entropy while numerical fluxes should dissipate entropy. We refer to  this as local entropy stability and in practice many numerical methods have demonstrated that this is necessary for numerical stability \cite{giraldo2020introduction, gassner2016well, morinishi2010skew, waruszewski2022entropy}.

As with the energy analysis in section 2.4, we first show that the evolution of the elemental buoyancy variance, $\zeta = \tfrac{1}{2}hb^2$, satisfies a conservation law before proving global conservation. The functional derivatives of $\zeta$ w.r.t. the prognostic variables are
\begin{equation}
    \dfrac{\partial \zeta}{\partial \vb{u}} = 0\text{,}\quad \dfrac{\partial \zeta}{\partial h} = -\tfrac{1}{2}b^2\text{,}\quad \dfrac{\partial \zeta}{\partial hb} = b\text{,}
\end{equation}
multiplying these by \eqref{eq:mass}--\eqref{eq:buoyancy} gives the evolution of $\zeta$ as \begin{equation}
\begin{split}
    \zeta_t &= -\tfrac{1}{2}b^2 h_t + b(hb)_t \\
    &= \tfrac{1}{2}b^2 \nabla \cdot \vb{F} - b\nabla \cdot \vb{B} \\ 
    &= \tfrac{1}{2}b^2 \nabla \cdot \vb{F} - \tfrac{1}{2}b^2\nabla \cdot \vb{F} - \tfrac{1}{2}\vb{B}\cdot \nabla b - \tfrac{1}{2} b\nabla \cdot \vb{B} \\ 
    &= - \tfrac{1}{2}\vb{B}\cdot \nabla b - \tfrac{1}{2} b\nabla \cdot \vb{B} \\ 
    &= -\tfrac{1}{2} \nabla \cdot b\vb{B}\text{,}
\end{split}
\label{eq:bvar-evo_0}
\end{equation}
with the required conservation law form
\begin{equation}
\begin{split}
    \zeta_t + \nabla \cdot \left(\frac{b\vb{B}}{2}\right) =0\text{.}
\end{split}
\label{eq:bvar-evo}
\end{equation}
 Note that this result \eqref{eq:bvar-evo} relies on the chain rule $\nabla\cdot \vb{B}=b\nabla \cdot \vb{F} + \vb{B}\cdot \nabla b$ which we can be difficult to mimic discretely with a DG-SEM scheme because of under integration and discontinuous function spaces. This motivates the operator splitting presented in the following section. 

Global conservation follows from the local conservation law, which we present in the following theorem.
\begin{theorem}\label{theo:conservation_of_entropy}
Consider the TRSW \eqref{eq:momentum}--\eqref{eq:buoyancy} on a 2D manifold $\Omega$ embedded in $R^3$ with periodic boundary conditions. At time $t \ge 0$, let the total buoyancy variance be denoted by $\mathcal{A}(t) = \int_\Omega{\zeta\Omega}$. We have
\begin{align}
    \frac{d \mathcal{A}}{dt} = 0, \quad \forall t\ge 0.
\end{align}

\end{theorem}
\begin{proof}
Integrating the buoyancy variance evolution equation \eqref{eq:bvar-evo} in space over the domain $\Omega$ gives
\begin{equation}
    \frac{d \mathcal{A}}{dt} = \frac{d }{dt} \int_\Omega{\zeta d\Omega} = -\int_\Omega{\tfrac{1}{2}\nabla \cdot b\vb{B}d\Omega} = -\int_{\partial\Omega}{\tfrac{1}{2}b\vb{B}\cdot \vb{n}dl} = 0\text{,}
\end{equation}
where $\vb{n}$ is the unit normal of the boundary. The final step results from cancellation across periodic boundaries. The proof is complete.
\end{proof}

We now prove that $\zeta$ is convex, and hence it is an entropy function.
\begin{theorem}
    The elemental buoyancy variance, $\zeta=\tfrac{1}{2}hb^2$, is a convex combination of the prognostic variables $h$ and $hb$.
\end{theorem}
\begin{proof}
    As the functional derivative of $\zeta$ w.r.t. $\vb{u}$ is $0$ we need only considers the Hessian of $\zeta$ w.r.t. $h$ and $hb$, giving
    \begin{equation}
        H_\zeta = \frac{1}{h}\begin{pmatrix}
            b^2 &  -b\\ 
            -b & 1
        \end{pmatrix}\text{,}
    \end{equation}
    with the characteristic equation
    \begin{equation}
        \lambda \big(\lambda - \frac{1+b^2}{h}\big) = 0\text{,}
    \end{equation}
    and eigenvalues $\lambda = 0, \dfrac{1+b^2}{h}$. Therefore $H_\zeta$ is positive semi-definite and $\zeta$ is convex.
\end{proof}

This analysis is quite general and applies to any flow with a mass continuity equation. 

\begin{remark}
We note that the entropy functional $\zeta=\tfrac{1}{2}hb^2$ controls only the fluid height $h$ and the buoyancy $hb$, but it does not control the flow velocity $\mathbf{u}$. However, numerical experiments performed later in this study demonstrate that an entropy stable numerical method is critical to keep numerical simulations accurate and stable for sufficiently long times.
\end{remark}

\begin{remark}
    For any fluid flow with a mass continuity equation, the variance $\rho q^2$ of any advected quantity $q$, where $\rho$ is density, is an entropy function if $\rho q$ is chosen to be a prognostic variable.
\end{remark}
This implies that our discrete formulation of the mass and mass-weighted buoyancy equation can be applied to conservatively and stably advect any quantity. As an example, potential enstrophy is an entropy of the shallow water equations with height and vorticity as prognostic variables, so our method could be used to advect vorticity while conserving potential enstrophy.

\subsection{Split form}

As mentioned in the previous section, the entropy conservation proof crucially relies on $\nabla \cdot \vb{B}=b\nabla \cdot \vb{F} + \vb{B}\cdot \nabla b$ which we cannot discretely mimic due to under-integration and discontinuous function spaces in DG-SEM. Following previous work on operator splitting for computational fluid dynamics \cite{gassner2016well, pirozzoli2010generalized, ducros2000high, morinishi2010skew, gassner2016split}, we equivalently reformulate the mass-weighted buoyancy evolution equation as
\begin{equation}
    (hb)_t + \tfrac{1}{2}\nabla \cdot \vb{B} + \tfrac{1}{2}\big(b\nabla \cdot \vb{F} + \vb{F} \cdot \nabla b\big) = 0\text{.}
    \label{eq:b-split}
\end{equation}
This introduces the required splitting directly into the mass-weighted buoyancy evolution equation and enables the entropy conservation proof to proceed without invoking the chain rule. The working in \eqref{eq:bvar-evo} now becomes
\begin{equation}
\begin{split}
    \zeta_t &= -\tfrac{1}{2}b^2 h_t + b(hb)_t \\
    &= \tfrac{1}{2}b^2 \nabla \cdot \vb{F} - \tfrac{1}{2}b^2\nabla \cdot \vb{F} - \tfrac{1}{2}b\vb{F}\cdot \nabla b - \tfrac{1}{2} b\nabla \cdot \vb{B} \\ 
    &= - \tfrac{1}{2}\vb{B}\cdot \nabla b - \tfrac{1}{2} b\nabla \cdot \vb{B}\text{,}
\end{split}
\end{equation}
with
\begin{equation}
\begin{split}
    \zeta_t + \tfrac{1}{2}\vb{B}\cdot \nabla b + \tfrac{1}{2} b\nabla \cdot \vb{B} =0\text{.}
\end{split}
\end{equation}
Entropy conservation follows from integration by parts over a periodic domain
\begin{equation}
    Z_t = \int_{\Omega}{\zeta_td\Omega} = -\tfrac{1}{2}\int_{\Omega}{(\vb{B}\cdot \nabla b + b\nabla \cdot \vb{B})d\Omega} = -\int_{\partial\Omega}{\tfrac{1}{2} b\vb{B}\cdot \vb{n} dl} = 0\text{.}
\end{equation}

The splitting of \eqref{eq:buoyancy} complicates the proof of energy conservation and necessitates additional operator splitting in the velocity equation. Recalling that $\frac{\partial e}{\partial hb} = \tfrac{1}{2}h$, with \eqref{eq:b-split} the energy exchange associated with the  mass-weighted buoyancy equation is expressed as
\begin{equation}
    \tfrac{1}{2}h(hb)_t = -\tfrac{1}{4}h \big(\nabla \cdot \vb{B} + b\nabla \cdot \vb{F} + \vb{F} \cdot \nabla b \big)\text{,}
\end{equation}
which is balanced with the energy exchange corresponding to this term in the
\begin{equation}
    -\tfrac{1}{2}b\vb{F}\cdot \nabla h
\end{equation}
kinetic energy evolution equation $\frac{\partial e}{\partial \vb{u}}\cdot \vb{u}_t= \vb{F}\cdot \vb{u}_t$.
In contrast to the energy conservation proof in section 2.4, with \eqref{eq:b-split} an energy conservation proof would now require the chain rule. This motivates splitting $b\nabla h$ in \eqref{eq:momentum} as
\begin{equation}
    b\nabla h = \tfrac{1}{2}b\nabla h -\tfrac{1}{2}(\nabla(hb) - h\nabla b)\text{,}
\end{equation} which again enables the energy conservation proof to proceed without the chain rule. The split-form of the TRSW considered by our DG method in the coming sections can now be formulated as
\begin{equation}
	\vb{u}_t + \omega\vb{k}\cross\vb{u} + \nabla G +\tfrac{1}{4}(b\nabla h + \nabla(hb) - h\nabla b)= 0\text{,}
	\label{eq:momentum1}
\end{equation}
\begin{equation}
	h_t + \nabla \cdot \vb{F} = 0\text{,}
	\label{eq:mass1}
\end{equation}
\begin{equation}
	(hb)_t + \tfrac{1}{2}\big(\nabla \cdot \vb{B} + b\nabla \cdot \vb{F} + \vb{F} \cdot \nabla b\big) = 0\text{.}
	\label{eq:buoyancy1}
\end{equation}
with
\begin{equation}\label{eq:fluxes_and_vorticity1}
	\omega = \vb{k} \cdot \nabla \cross \vb{u} + f\text{,} \quad \vb{B} = hb\vb{u}\text{,} \quad \vb{F} = h\vb{u}\text{,} \quad G = \tfrac{1}{2}\vb{u}\cdot\vb{u} + \tfrac{1}{2}ghb\text{.}
\end{equation}
\begin{remark}
    It is significantly noteworthy that the theoretical results proved in Theorems
\ref{theo:conservation_of_mass_bouyancy}-\ref{theo:conservation_of_entropy} are also applicable to the split-form of the TRSW \eqref{eq:momentum1}--\eqref{eq:buoyancy1}. However, the corresponding analysis circumvents the chain rule and requires only integration by parts which we can mimic discretely in a DG frame-work. 
\end{remark}

\section{The DG method}

In this section we briefly describe our discrete spaces before presenting our discrete method. 

\subsection{Spaces}
We decompose the domain into quadrilateral elements $\Omega^m$. For each element we define an invertible map to the reference square $r(\vb{x};m): \Omega^m \rightarrow [-1, 1]^2$, and approximate solutions by polynomials of order $m$ in this reference square. We use a computationally efficient tensor product Lagrange polynomial basis with interpolation points collocated with Gauss-Lobatto-Legendre (GLL) quadrature points. Using this basis the polynomials of order $p$ on the reference square can be defined as
\begin{equation}
    P_p := \text{span}_{i, j=1}^{p+1} l_i(\xi) l_j(\eta)\text{,}
\end{equation}
where $(\xi, \eta) \in [-1, 1]^2$ and $l_i$ is the 1D Lagrange polynomial which interpolates the $i^{th}$ GLL node. We define a discontinuous scalar space $S$ which contains the depth $D^h$ as
\begin{equation}
    S = \{\phi \in L^2(\Omega) : \phi\left(r^{-1}\left(\xi, \eta; m\right)\right) \in P_p, \forall m \} \text{.}
\end{equation}
Note that within each element functions $\phi \in S$ can be expressed as
\begin{equation}
    \phi(\vb{x}) = \sum_{i,j=1}^{p+1}{\phi_{ijm}l_i(\xi)l_j(\eta)}, \forall \vb{x} \in \Omega^m\text{,}
\end{equation}
where $\xi, \eta = r(\vb{x}; m)$, $\phi_{ijm} = \phi(r^{-1}(\xi_i, \eta_j;m))$.

We then define a discontinuous vector-valued function space as a tensor product of $S$. Let $\vb{v}_1(x,y,z)$ and $\vb{v}_2(x,y,z)$ be any independent set of vectors which span the tangent space of the manifold at each point $(x,y,z)$, then the discontinuous vector space $V$ containing the velocity $\vb{u}$ can be defined as
\begin{equation}
    V = \{\vb{w} : \vb{w} \cdot \vb{v}_i \in S, i=1,2 \}\text{.}
\end{equation}
We note that $V$ is independent of the particular choice of $\vb{v}_i$ however two convenient choices are the contravariant and covariant basis vectors. 

\subsection{Operators}

Before defining our discrete operators we first introduce the covariant vectors
\begin{equation}
    \vb{g}_1 = \dfrac{\partial \vb{x}}{\partial \xi}, \quad \vb{g}_2 = \dfrac{\partial \vb{x}}{\partial \eta}\text{,}
\end{equation}
and the contravariant vectors
\begin{equation}
    \vb{g}^1 = \nabla \xi, \quad \vb{g}_2 = \nabla \eta \text{.}
\end{equation}
With these, the determinant of the Jacobian of $r(x; m)$ is given by
\begin{equation}
    J = |\vb{g}_1 \times \vb{g}_2|\text{.}
\end{equation}

Following \cite{taylor2010compatible}, we use the covariant and contravariant vectors to discretise the divergence, gradient, and curl as
\begin{equation}
    \nabla \cdot \vb{w} = \dfrac{1}{J}\bigg(\dfrac{\partial Jw^1}{\partial \xi} + \dfrac{\partial Jw^2}{\partial \eta} \bigg)\text{,}
    \label{eq:div}
\end{equation}
\begin{equation}
    \nabla \phi = \dfrac{\partial \phi}{\partial \xi} \vb{g}^1 + \dfrac{\partial \phi}{\partial \eta} \vb{g}^2\text{,} 
    \label{eq:grad}
\end{equation}
\begin{equation}
    \nabla \cross \vb{w} = \dfrac{1}{J}\bigg(\dfrac{\partial w_2}{\partial \eta} - \dfrac{\partial w_1}{\partial \xi} \bigg)\vb{k}\text{,} 
    \label{eq:curl}
\end{equation}
\begin{equation}
    \nabla \cross \phi \vb{k} = \dfrac{1}{J}\dfrac{\partial \phi}{\partial \eta}\vb{g}_1 - \dfrac{1}{J}\dfrac{\partial \phi}{\partial \xi} \vb{g}_2 \text{,} 
    \label{eq:curl2}
\end{equation}
where we have only used $\nabla \cross$ on quantities either parallel or orthogonal to $\vb{k}$ to simplify the exposition.

\subsection{Integrals}

We approximate integrals over the elements by using GLL quadrature. This defines a discrete element inner product
\begin{equation}
    \left\langle f, g \right\rangle_{\Omega^m} := \sum_{i, j=1}^{p+1}{w_i w_j J_{ij} f_{ij}g_{ij}}\text{,}
\end{equation}
where $w_i$ are the 1D quadrature weights. This approximates the integral
\begin{equation}
    \int_{\Omega^m}{fg d\Omega^m} \approx \left\langle f, g \right\rangle_{\Omega^m}\text{.}
\end{equation}
The discrete global inner product is defined simply as $\left\langle f, g \right\rangle_{\Omega} = \sum_m{\left\langle f, g \right\rangle_{\Omega^m}}$. Similarly for vectors 
\begin{equation}
    \left\langle \vb{f}, \vb{g} \right\rangle_{\Omega^m} := \sum_{i, j=1}^{p+1}{w_i w_j J_{ij} \vb{f}_{ij} \cdot \vb{g}_{ij}}\text{.}
\end{equation}
We also approximate element boundary integrals using GLL quadrature
\begin{equation}
    \int_{\partial \Omega^m}{\vb{f}\cdot\vb{n}dl} \approx \left\langle \vb{f}, \vb{n}\right\rangle_{\partial\Omega^m}\text{,}
\end{equation}
where 
\begin{equation}
\begin{split}
    \left\langle \vb{f}, \vb{n}\right\rangle_{\Omega^m} := \sum_{j}{w_j\big(|g_1|_{1j} \vb{f}_{1j}\cdot \vb{n}_{1j} + |g_1|_{(p+1)j} \vb{f}_{(p+1)j}\cdot \vb{n}_{(p+1)j}} \\ {+ |g_2|_{j1}\vb{f}_{j1}\cdot \vb{n}_{j1} + |g_2|_{j(p+1)}\vb{f}_{j(p+1)}\cdot \vb{n}_{j(p+1)}\big)}\text{,}
\end{split}
\end{equation}
and $\vb{n}$ is the unit normal of the element boundary $\partial \Omega^m$. Similarly, we use $\vb{t}$ to denote the tangent vectors to the element boundaries, and we also use the boundary jump notation
\begin{equation}
    \average{a} = \tfrac{1}{2}(a^+ + a^-)\text{,}
\end{equation}
\begin{equation}
    \jump{a} = a^+ - a^-\text{,}
\end{equation}
where $+$ and $-$ superscripts are arbitrary and denote quantities on opposite sides of the element boundary.

With the discrete setup defined, we now present the following lemma which is crucial to our discrete stability and conservation proofs.
\begin{lemma}
The discrete spaces satisfy an element-wise SBP property.
\begin{equation}
    \left\langle \phi, \nabla_d\cdot \vb{w}\right\rangle_{\Omega^m} + \left\langle \nabla_d\phi, \vb{w}\right\rangle_{\Omega^m} = \left\langle \phi\vb{w}, \vb{n}\right\rangle_{\partial\Omega^m}\text{,}\; \forall \phi \in S, \vb{w} \in V\text{.}
\end{equation}
\end{lemma}
\begin{proof}
    See \cite{taylor2010compatible}.
\end{proof}

\subsection{Discrete formulation}

We use the strong form DG-SEM, which has shown to be equivalent to the weak form \cite{kopriva2010quadrature}. Our discrete method is
\begin{equation}
\begin{split}
	\left\langle \vb{w}, \vb{u}_t\right\rangle_{\Omega^m} + \left\langle \vb{w}, \omega \vb{k}\cross \vb{u}^{h} + \nabla_d G + \tfrac{1}{4}(b\nabla_d h + \nabla_d (hb) - h\nabla_d b)\right\rangle_{\Omega^m}  \\
+ \left\langle \vb{w} \cdot \vb{n}, \tfrac{1}{2}\hat{b}(\average{h}-h) + (\hat{G}-G)\right\rangle_{\partial\Omega^m} = 0\text{,}\;\forall \vb{w}\in V\text{,}
\end{split}
\label{eq:dg-velocity}
\end{equation}

\begin{equation}
	\left\langle \phi, h_t\right\rangle_{\Omega^m} + \left\langle \phi, \nabla_d \cdot \vb{F}\right\rangle_{\Omega^m} + \left\langle \phi, \big(\hat{\vb{F}}-\vb{F}\big)\cdot \vb{n}\right\rangle_{\partial\Omega^m} = 0\text{,}\;\forall \phi\in S\text{,}
 \label{eq:dg-depth}
\end{equation}
\begin{equation}
\begin{split}
	\left\langle \psi, (hb)_t\right\rangle_{\Omega^m} &+ \left\langle \psi, \tfrac{1}{2}(b\nabla_d \cdot \vb{F} + \vb{F}\cdot \nabla_d b + \nabla_d \cdot \vb{B})\right\rangle_{\Omega^m} \\ 
 & + \left\langle \psi, \big(\hat{\vb{B}}-\vb{B}\big)\cdot \vb{n}\right\rangle_{\partial\Omega^m} = 0\text{,}\;\forall \psi\in S\text{,}
 \end{split}
 \label{eq:dg-b}
\end{equation}
where the numerical fluxes are defined for the scalars $\alpha, \beta$ as
\begin{equation}
    \hat{G} = \average{ G } + \alpha\jump{\vb{F}}\cdot{\vb{n}^{+}}\text{.}
\end{equation}
\begin{equation}
    \hat{\vb{F}} = \average{ \vb{F} }\text{,}
\end{equation}
\begin{equation}
    \hat{\vb{B}}= \hat{b}\average{ \vb{F} }\text{,}
\end{equation}
\begin{equation}
    \hat{b} = \average{b} + \beta \jump{b}\text{,}
\end{equation}
and the discrete vorticity is 
\begin{equation}
    \left\langle \phi, \omega\right\rangle_{\Omega^m} = \left\langle -\nabla_d \cross \phi\vb{k}, \vb{u}\right\rangle_{\Omega^m} + \left\langle \phi, f\right\rangle_{\Omega^m} + \left\langle \phi, \average{\vb{u}} \cdot \vb{t}\right\rangle_{\partial\Omega^m}\;\forall \phi\in S\text{.}
    \label{eq:dg-vort}
\end{equation}
This builds upon the DG RSW method of \cite{ricardo2023conservation}, with the main contributions being our novel numerical fluxes $\hat{\vb{B}}$ and $\hat{b}$, and the simultaneous splitting of both the divergence operator in \eqref{eq:dg-b} and the $b\nabla h$ term in \eqref{eq:dg-velocity}. Both the fluxes and the particular operator splitting are critical for our conservation and stability proofs. 

% For function spaces with the SBP property, splitting $\nabla\cdot\vb{B}$ ensures the volume terms conserve entropy, while splitting $b\nabla h$ balances energy exchanges with the new split form of $\nabla \cdot \vb{B}$.

Our method is semi-discretely energy and entropy conserving for the choice $\alpha=\beta=0$ which we refer to as the conservative method. The other particular choice of parameters we use is
\begin{equation}\label{eq:alpha_beta}
    \alpha = \frac{1}{2}\max\left(\frac{c^+}{h^+}, \frac{c^-}{h^-}\right)\text{,}\quad \beta = \tfrac{1}{2}\text{sign}(\average{\vb{F}}\cdot\vb{n}^-)\text{,}
\end{equation}
where $c^\pm =|\vb{u}^{\pm}| + \sqrt{gh^{\pm}}$ is the fastest wave speed on either side of the boundary. This parameter choice semi-discretely dissipates entropy and kinetic energy and we refer to this as the dissipative method. As in \cite{ricardo2023conservation}, this choice of $\alpha$ ensures that $\hat{G}$ has the correct units and reduces to a Rusanov flux for the TRSW linearised around a constant state with zero mean flow. The choice of $\beta$ simply upwinds $b$ and is equivalent to 
\begin{equation}
\hat{b} = 
    \begin{cases}
    b^-,& \text{if $\average{\vb{F}} \cdot\vb{n}^- \geq 0$}\\
    b^+,& \text{otherwise.}
    \end{cases}
\end{equation}
We did not a priori choose an unpwinded $\hat{b}$, it was the only option that the authors could find which was consistent, entropy dissipating, and did not spurious generate energy. The requirement of balanced kinetic and potential energy exchanges necessitates that $\vb{\hat{B}} = \hat{b}\average{\vb{F}}$. The role of $\hat{\vb{F}}$ in both kinetic energy and entropy exchanges requires the use of a centred flux. Finally, consistency and entropy dissipation dictate that $\hat{b}$ be upwinded. Consequently, an upwinded $\hat{b}$ appears to be a very natural choice.

\section{Semi-discrete conservation and numerical stability}
In this section we prove some discrete conservation properties and establish the numerical stability of the numerical method \eqref{eq:dg-velocity}--\eqref{eq:dg-vort}. The main results of this paper are presented in this section, including conservation of mass and mass-weighted buoyancy at the discrete level, semi-discrete energy conservation, and numerical entropy stability.

\subsection{Conservation of mass and mass-weighted buoyancy}
Here, we will show that conservation of mass and mass-weighted buoyancy directly follows from the SBP property. First, we define the elemental discrete mass and mass-weighted buoyancy as $\mathcal{M}^m=\left\langle h \right\rangle_{\Omega^m}$ and $\mathcal{B}^m=\left\langle hb \right\rangle_{\Omega^m}$, and the global mass and mass-weighted buoyancy as $\mathcal{M} = \sum_{m}{\mathcal{M}^m}$ and $\mathcal{B} = \sum_m{\mathcal{B}^m}$.     

\begin{theorem}
    The semi-discrete method \eqref{eq:dg-velocity}-\eqref{eq:dg-vort} conserves mass and mass-weighted buoyancy in a periodic domain, that is $\mathcal{M}_t = \mathcal{B}_t = 0$.
\end{theorem}
\begin{proof}
    Substituting the test function $\phi=1$ into the mass equation \eqref{eq:dg-depth} and using SBP gives
    \begin{equation}
        \mathcal{M}^m_t = -\left\langle 1, \nabla \cdot \vb{F} \right\rangle_{\Omega^m} - \left\langle 1, \hat{\vb{F}} - \vb{F} \right\rangle_{\partial\Omega^m} =-\left\langle \phi, \hat{\vb{F}} \right\rangle_{\partial\Omega^m} \text{,}
    \end{equation}
    as the domain is periodic and $\hat{\vb{F}}$ is continuous by construction the contributions from neighbouring elements cancel giving
    \begin{equation}
        \mathcal{M}_t = -\sum_m{\left\langle 1, \hat{\vb{F}} \right\rangle_{\partial\Omega^m}} = 0\text{,}
    \end{equation}
    as required.

    Substituting $\psi=1$ into the mass-weighted buoyancy equation \eqref{eq:dg-b} gives
    \begin{equation}
        \mathcal{B}^m_t = \left\langle 1, (hb)_t \right\rangle_{\Omega^m} = -\left\langle \tfrac{1}{2}, \nabla \cdot \vb{B} + b\nabla \cdot \vb{F} + \vb{F}\cdot \nabla b\right\rangle_{\Omega^m} - \left\langle 1, \hat{\vb{B}}-\vb{B}\right\rangle_{\partial\Omega^m}\text{,}
    \end{equation}
    using SBP yields
    \begin{equation}
        \mathcal{B}^m_t = - \left\langle 1, \hat{\vb{B}}\right\rangle_{\partial\Omega^m}\text{,}
    \end{equation}
    and again by periodicity and continuity of $\hat{\vb{B}}$
    \begin{equation}
        \mathcal{B}_t = -\sum_m{\left\langle 1, \hat{\vb{B}} \right\rangle_{\partial\Omega^m}} = 0\text{,}
    \end{equation}
    as required.
\end{proof}
We note that this proof also demonstrates that our scheme locally conserves these quantities.

\subsection{Conservation of entropy}

Here we show that our conservative method semi-discretely conserves entropy, and that both the conservative and dissipative methods are semi-discretely entropy stable. To achieve fully discrete entropy stability, we employ a strong stability preserving (SSP) time integrator \cite{shu1988efficient}. This analysis, similar to the continuous case, is quite general and can be utilized to derive a discretely conservative and stable advection scheme for any advected quantity. First, we define the elemental and global entropy as 
\begin{equation}
    \mathcal{A}^m=\left\langle 1,\frac{(hb)^2}{h}\right\rangle_{\Omega^m}\text{,}\quad \mathcal{A} = \sum_m{\mathcal{A}^m}\text{.}
\end{equation}
 
\begin{theorem}
    The semi-discrete method \eqref{eq:dg-velocity}-\eqref{eq:dg-vort} is entropy stable in a periodic domain, that is $\mathcal{A}_t \leq 0$. Additionally, with a centred $\hat{b} = \average{b}$ the semi-discrete method conserves entropy $\mathcal{A}_t = 0$.  
\end{theorem}
\begin{proof}
    Mirroring the continuous analysis in section 2.4, with the semi-discrete approximation the time derivative of entropy is
    \begin{equation}
        \mathcal{A}^m_t = -\left\langle \tfrac{1}{2}b^2, h_t\right\rangle_{\Omega^m} +\left\langle b, (hb)_t\right\rangle_{\Omega^m} \text{.}
    \end{equation}
Substituting the discrete mass and mass-weighted buoyancy equations \eqref{eq:dg-depth}-\eqref{eq:dg-b} this becomes
\begin{equation*}
\begin{split}
    \mathcal{A}^m_t &= \left\langle \tfrac{1}{2}b^2,\nabla_d  \cdot\vb{F}\right\rangle_{\Omega^m} - \left\langle\tfrac{1}{2}b^2,\nabla_d \cdot  \vb{F}\right\rangle_{\Omega^m} - \left\langle\tfrac{1}{2}\vb{B}, \nabla b\right\rangle_{\Omega^m} - \left\langle\tfrac{1}{2}b\nabla_d,  \vb{B}\right\rangle_{\Omega^m} \\
    &-\left\langle b(\hat{\vb{B}} - \vb{B}) - \tfrac{1}{2}b^2(\hat{\vb{F}} - \vb{F}), \vb{n}\right\rangle_{\partial\Omega^m} \\
    &=- \left\langle\tfrac{1}{2}\vb{B}, \nabla b\right\rangle_{\Omega^m} - \left\langle\tfrac{1}{2}b\nabla_d,  \vb{B}\right\rangle_{\Omega^m} -\left\langle b(\hat{\vb{B}} - \vb{B}) - \tfrac{1}{2}b^2(\hat{\vb{F}} - \vb{F}), \vb{n}\right\rangle_{\partial\Omega^m}
    \text{,}
\end{split}
\end{equation*}
and using SBP then yields 
\begin{equation}
    \mathcal{A}^m_t = -\left\langle b\hat{\vb{B}} - \tfrac{1}{2}b^2\hat{\vb{F}}, \vb{n}\right\rangle_{\partial\Omega^m}\text{.}
\end{equation}

To ensure that the numerical method is  entropy stable, the jump in the numerical entropy flux across element boundaries must be negative semi-definite, that is
\begin{equation}
    \jump{-b\hat{\vb{B}} + \tfrac{1}{2}b^2\hat{\vb{F}}}\cdot \vb{n}^- \le 0\text{.}
\end{equation}
Expanding the jump condition using $\jump{ab} = \average{a}\jump{b} + \jump{a}\average{b}$ gives
\begin{equation}
\begin{split}
     \jump{-b\hat{\vb{B}} + \tfrac{1}{2}b^2\hat{\vb{F}}}\cdot \vb{n}^-=\jump{b}(-\hat{\vb{B}} + \average{b}\hat{\vb{F}})\cdot \vb{n}^- = -\beta \jump{b}^2 \average{\vb{F}}\cdot \vb{n}^-.
\end{split}
\end{equation}
Thus if $\beta =0$, then the jump vanishes and we have 
$$
\jump{-b\hat{\vb{B}} + \tfrac{1}{2}b^2\hat{\vb{F}}}\cdot \vb{n}^- =0 \implies \mathcal{A}_t =0.
$$
When $\beta = \tfrac{1}{2}\text{sign}(\average{\vb{F}}\cdot\vb{n}^-)$ we have 
$$
\jump{-b\hat{\vb{B}} + \tfrac{1}{2}b^2\hat{\vb{F}}}\cdot \vb{n}^- =- \frac{1}{2} \jump{b}^2 |\average{\vb{F}}\cdot \vb{n}^-|\le 0 \implies \mathcal{A}_t \le 0.
$$

\end{proof}

\subsection{Conservation of energy}

Here we show our semi-discrete method is energy stable, and that for a centred $\hat{G} = \average{G}$ it conserves energy. It is important to note that unlike the vector invariant SWE \cite{ricardo2023conservation}, energy is not a convex function of the prognostic variables. Therefore  there is no guarantee of discrete energy stability with an SSP time integrator.

We define the discrete elemental and global energy as $\mathcal{E}^m = \left\langle \tfrac{1}{2}h|\vb{u}|^2 + \tfrac{1}{2}h(hb)\right\rangle_{\Omega^m}$ and $\mathcal{E} = \sum_m{\mathcal{E}^m}$.

\begin{theorem}
    The semi-discrete method \eqref{eq:dg-velocity}-\eqref{eq:dg-vort} is energy stable in a periodic domain, that is $\mathcal{E}_t \leq 0$. Additionally, with a centred $\hat{G} = \average{G}$ the semi-discrete method conserves energy $\mathcal{E}_t = 0$.
\end{theorem}
\begin{proof}
    Mirroring the continuous analysis in section 2.3, with the semi-discrete approximation the time derivative of the energy is
    \begin{equation}
        \mathcal{E}^m_t = \left\langle\frac{\partial e}{\partial \vb{u}}, \vb{u}_t\right\rangle_{\Omega^m} + \left\langle \frac{\partial e}{\partial h}, h_t\right\rangle_{\Omega^m} + \left\langle\frac{\partial e}{\partial hb}, (hb)_t\right\rangle_{\Omega^m}\text{.}
    \end{equation}
Substituting $\phi=\frac{\partial e}{\partial h}=G$ and $\psi=\frac{\partial e}{\partial hb}=\tfrac{1}{2}h$ into \eqref{eq:dg-depth}-\eqref{eq:dg-b} and using SBP yields
\begin{equation}
    \left\langle \frac{\partial e}{\partial h}, h_t\right\rangle_{\Omega^m} = \left\langle \nabla_d G, \vb{F}\right\rangle_{\Omega^m} - \left\langle G, \hat{\vb{F}}\cdot \vb{n}\right\rangle_{\partial\Omega^m},
\end{equation}
\begin{equation}
\begin{split}
    \left\langle \frac{\partial e}{\partial hb}, (hb)_t\right\rangle_{\Omega^m} &= \left\langle \tfrac{1}{2}\nabla_d h, \frac{1}{2}\vb{B}\right\rangle_{\Omega^m}+\left\langle \tfrac{1}{2}\nabla_d (hb), \frac{1}{2}\vb{F}\right\rangle_{\Omega^m} \\ &- \left\langle \tfrac{1}{2}h, \frac{1}{2}\vb{F}\cdot\nabla_d b\right\rangle_{\Omega^m}-\left\langle \tfrac{1}{2}h, \hat{\vb{B}}\cdot \vb{n}\right\rangle_{\partial\Omega^m}.
\end{split}
\end{equation}
Recalling $\frac{\partial e}{\partial \vb{u}}=\vb{F}$ and noting $\vb{F}\cdot \omega \vb{k}\cross \vb{u}=h\omega \vb{u}\cdot \vb{k}\cross \vb{u}=0$ for collocated methods, we see that the internal terms of the energy evolution equation cancel yielding an evolution equation for $\mathcal{E}^m$ in terms of a numerical flux
\begin{equation}
    \mathcal{E}^m_t + \left\langle \boldsymbol{\mathcal{E}}_{flux}\cdot \vb{n}\right\rangle_{\partial\Omega^m} = 0,
\end{equation}
where
\begin{equation}
  \boldsymbol{\mathcal{E}}_{flux} = \vb{F}\hat{G} - \vb{F}G + G\hat{\vb{F}} + \tfrac{1}{2}\hat{b}\average{h}\vb{F} - \tfrac{1}{2}\hat{b}h\vb{F} + \tfrac{1}{2}h\hat{\vb{B}}.
\end{equation}
To ensure that the energy is stable the jump in the numerical flux across element boundaries must be be negative semi-definite
\begin{equation}
    -\jump{
\boldsymbol{\mathcal{E}}_{flux}
    }\cdot \vb{n}^+ \leq 0\text{.}
\end{equation}
Recalling the definitions of the numerical fluxes 
\begin{equation}
    \hat{G} = \average{G} + \alpha\jump{\vb{F}}\cdot \vb{n}^+, \quad \hat{\vb{F}}=\average{\vb{F}}, \quad \hat{\vb{B}}=\hat{b}\hat{\vb{F}} = \hat{b}\average{\vb{F}},
\end{equation}
 and using $\jump{ab} = \jump{a}\average{b} + \average{a}\jump{b}$ yields
\begin{equation}
\begin{split}
    -\jump{
\boldsymbol{\mathcal{E}}_{flux}
    }\cdot \vb{n}^+ &= -\left( 
    \jump{\vb{F}}\hat{G} - \jump{\vb{F} G} + \average{\vb{F}}\jump{G}
    \right)\cdot \vb{n}^+ \\ & - \tfrac{1}{2}\left(\hat{b}\average{h}\jump{\vb{F}} -\hat{b}\jump{h\vb{F}}  + \jump{h}\hat{b}\average{\vb{F}}\right)\cdot \vb{n}^+ \\
    &=-\left( 
    \jump{\vb{F}}\hat{G} - \jump{\vb{F}}\average{G} - \average{\vb{F}}\jump{G} + \average{\vb{F}}\jump{G}
    \right)\cdot \vb{n}^+ \\ & - \tfrac{1}{2}\left(\hat{b}\average{h}\jump{\vb{F}} -\hat{b}\jump{h}\average{\vb{F}} -\hat{b}\average{h}\jump{\vb{F}} + \hat{b}\jump{h}\average{\vb{F}}\right)\cdot \vb{n}^+ \\
    &= -\left( 
    \jump{\vb{F}}\hat{G} - \jump{\vb{F}}\average{G}
    \right)\cdot \vb{n}^+ \\
    &= - \alpha \left(\jump{\vb{F}}\cdot \vb{n}^+\right)^2,
\end{split}
\end{equation}
and therefore $\alpha \geq 0 \implies \mathcal{E}_t \leq 0$ and $\alpha = 0 \implies \mathcal{E}_t = 0$, completing the proof.
\end{proof}

\subsection{Conservation of vorticity}

In contrast to the RSW, the TRSW only globally conserves vorticity and does not possess a flux form vorticity conservation equation. Here, we prove that the discrete vorticity is globally conserved.
% \todo[inline]{Is the  statement above true? I think by applying $\vb{k}\cdot\nabla \cross$ to \eqref{eq:momentum} it is possible to derive the continuity equation for vorticity
% $$
% \frac{\partial \omega}{\partial t} + \div{(\omega \bm{u})} =0
% $$
% }
\begin{theorem}
    The discrete method \eqref{eq:dg-velocity}-\eqref{eq:dg-vort} conserves total absolute vorticity $\mathcal{W} = \sum_m \left\langle 1, \omega\right\rangle_{\Omega^m}$.
\end{theorem}
\begin{proof}
Substituting $\phi=1$ into \eqref{eq:dg-vort} and summing over all elements yields the total vorticity
\begin{equation}
\begin{split}
     \mathcal{W} &= \sum_m\left\langle -\nabla_d \cross 1\vb{k}, \vb{u}\right\rangle_{\Omega^m} + \left\langle 1, f\right\rangle_{\Omega^m} + \left\langle 1, \average{\vb{u}} \cdot \vb{t}\right\rangle_{\partial\Omega^m} \\
    &= \left\langle 1, f\right\rangle_{\Omega}\text{,}
\end{split}
\end{equation}
where the first term is $0$ as the discrete curl of a constant is $0$, and the third term cancels when summing over elements as $\average{\vb{u}}$ is continuous across element boundaries by construction. $f$ is temporally constant and therefore vorticity is globally conserved.
\end{proof}

\subsection{Compatible buoyancy advection}

Compatible buoyancy advection is equivalent to showing that for the case $b=1$ the discrete mass-weighted buoyancy equation reduces to the discrete mass equation \cite{eldred2019quasi}. 
\begin{theorem}
    The discrete method \eqref{eq:dg-velocity}-\eqref{eq:dg-vort} compatibly advects buoyancy.
\end{theorem}
\begin{proof}
    We show that both the volume terms and numerical fluxes in \eqref{eq:dg-depth} and \eqref{eq:dg-b} are equivalent when $b=1$. Substituting $b=1$ into the volume terms of \eqref{eq:dg-b} yields
\begin{equation}
	\left\langle \phi, \tfrac{1}{2}(1\nabla_d \cdot \vb{F} + \vb{F}\cdot \nabla_d 1 + \nabla_d \cdot 1\vb{F})\right\rangle_{\Omega^m} = \left\langle \phi, \nabla_d \cdot \vb{F}\right\rangle_{\Omega^m}\text{.}
\end{equation}
Noting that $b=1$ everywhere $\implies \hat{b} = 1$, therefore $\hat{\vb{B}} = \hat{b}\hat{\vb{F}} = \hat{\vb{F}}$. Therefore \eqref{eq:dg-b} reduces to \eqref{eq:dg-depth} when $b=1$, and hence the buoyancy is compatibly advected.
\end{proof}

\section{Numerical Results}
In this section we present the results of numerical experiments which verify our theoretical analysis. For all experiments we use an equi-angular cubed sphere mesh \cite{ranvcic1996global}, an SSP-RK3 time integrator \cite{shu1988efficient}, and third-order polynomials. Unless otherwise mentioned, in our experiments we use an adaptive time step with a fixed CFL of 0.8, $CFL = \frac{\Delta t}{(2p + 1)Delta x}$ where $p$ is the polynomial order, $c$ is the fastest wave speed, and $\Delta x$ is the element spacing.

\subsection{Steady state convergence test}

We modified the Williamson test case 2 for steady geostrophic balance of the rotating shallow water equations on the sphere \cite{williamson1992standard} using the procedure described in \cite{eldred2019quasi} to obtain a steady state for the TRSW. The resulting initial condition is
\begin{equation}
    u = u_0 \cos(\theta)\text{,}
\end{equation}
\begin{equation}
    h = H - \tfrac{1}{g}\big(a f u_0 + \tfrac{1}{2}u_0^2\big)\sin^2(\theta)\text{,}
\end{equation}
\begin{equation}
    b = g\left(1 + c\tfrac{H}{h^2}\right)\text{,}
\end{equation}
where $\theta$ is latitude, $u$ is the zonal velocity, $a=6.37122\times 10^6$ m is the planetary radius, $g=9.80616 \text{ ms}^{-2}$, $u_0 = 38.61068 \text{ ms}^{-1}$ is the maximum velocity, $H = \tfrac{1}{g}2.94\times 10^4$ m is the maximum height, $c=0.05$, and $f = 2\times 7.292\times 10^{-5} \sin\theta \text{ s}^{-1}$.

Figure \ref{fig:steady_state_error} shows the error for the conservative and dissipative methods at day 5, which converge at order 3.4 and 3.8 respectively. While the orders of convergence are similar, the dissipative method is nearly an order of magnitude more accurate. This is likely due to the dissipative method damping high frequency waves close to the grid scale, which are known to be inaccurate for SEM and can trigger further errors in non-linear systems.

\begin{figure}[!hbtp]
\begin{center}
	\includegraphics[width=0.8\textwidth]{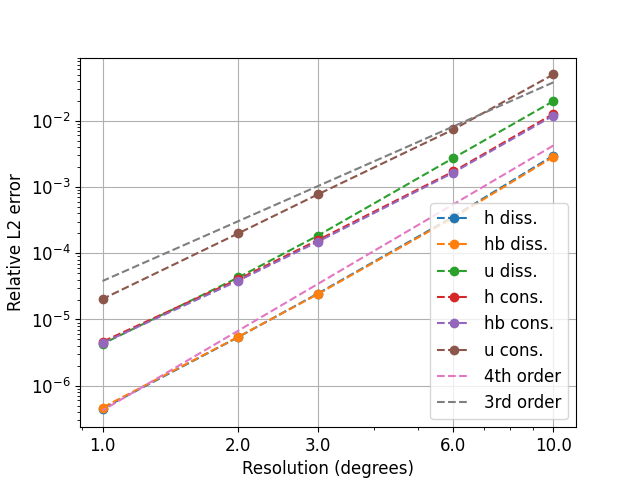}
 \caption{Steady state test case $L^2$ errors for the conservative and dissaptive methods at various resolutions. The resolution is taken as the average nodal spacing and results are shown at 5 days.}
 \label{fig:steady_state_error}
\end{center}
\end{figure}

\subsection{Modified Galewsky}

The Galewsky test case \cite{galewsky2004initial} is a RSWE test case case in which a steady state zonal jet perturbed by a Gaussian hill in the depth field triggers a shear flow instablity. The Gaussian hill induces gravity waves which excite a shear instability, causing the vorticity field to roll up into distinctive Kelvin-Helmholtz billows and the flow to become turbulent. We modify this test case by normalising the depth perturbation from the original Galewksy test and applying this to $b$.  This gives the initial condition
\begin{equation}
    u = \begin{cases}
        \dfrac{u_0}{\gamma_1} \exp\big((\theta-\theta_0)^{-1}(\theta-\theta_1)^{-1}\big)\text{, for } \theta_0 < \theta < \theta_1 \\ 
        0\text{, otherwise}
    \end{cases}
\end{equation}
\begin{equation}
    h = H + p(\lambda, \theta) - \dfrac{a}{g}\int_{-\tfrac{\pi}{2}}^\theta{ u(\theta') \bigg[f + \dfrac{\tan(\theta')}{a}u(\theta')\bigg]d\theta'}\text{,}
\end{equation}
\begin{equation}
    b = g + \tfrac{1}{120}p(\lambda, \theta)\text{,}
\end{equation}
\begin{equation}
    p(\lambda, \theta) = 120\cos(\theta) \exp(-(\gamma_2\lambda)^2)\exp(-\gamma_3^2(\theta - \theta_2))\text{,}
\end{equation}
where $p$ is the pertubation, $u$ is the zonal velocity, $\theta$ is latitude, $\lambda$ is longitude, $a=6.37122 \times 10^6$ m is the sphere's radius, $H = 10^4$ m is the maximum depth, $u_0 = 80 \text{ ms}^{-1}$ is the maximum velocity, $\gamma_1 = \exp(-4 (\theta_1 - \theta_0)^{-2})$  is a normalisation constant, $\gamma_2=3$, $\gamma_3=15$, $\theta_0 = \tfrac{\pi}{7}$, $\theta_1 = \tfrac{\pi}{2} - \theta_0$, $\theta_2 = \tfrac{\pi}{4}$, $f=2 \times 7.292 \times 10^{-5} \sin\theta \text{ s}^{-1}$, and $g=9.80616 \text{ ms}^{-2}$.

\subsubsection{Energy and entropy conservation}

To verify that our conservative method semi-discretely conserve energy and entropy we run the modified Galewsky test case at a coarse resolution of $6\times5\times5$ elements for decreasing time steps. Figure \ref{fig:energy-error} shows the errors at day 10 which converge at 3rd order, verifying the semi-discrete energy and entropy conservation of our scheme.

\begin{figure}[!hbtp]
\begin{center}
	\includegraphics[width=0.8\textwidth]{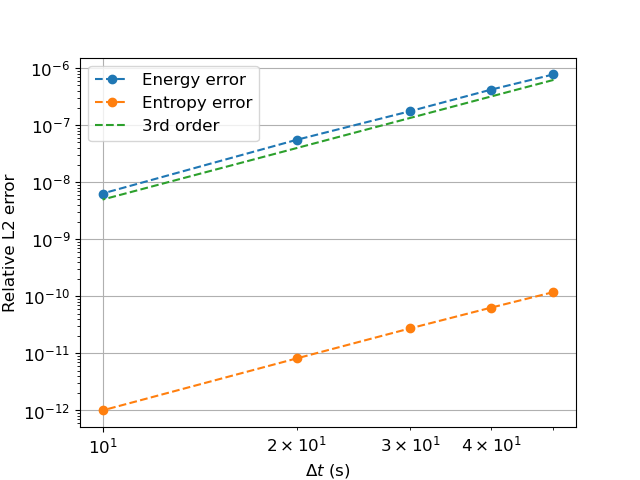}
 \caption{Relative energy and entropy conservation error at day 10 of the modified Galewsky test case for the conservative scheme with varying timesteps and $6\times5\times 5$ third order elements.}
  \label{fig:energy-error}
\end{center}
\end{figure}

\subsubsection{High resolution simulation}

We performed the modified Galewksy experiment at a high resolution of $6\times64\times64$ to examine the discrete methods ability to simulate geostrophic turbulence. Figure \ref{fig:conservation} shows the conservation errors, as expected for both methods the mass and mass-weighted buoyancy are conserved up to machine precision, and the entropy is discretely stable. Interestingly the energy also appears to be discretely stable for this test case. We note that even without dissipation the conservative method runs stably for the full 20 day period. Figures \ref{fig:galewsky-high-res-vort-7}-\ref{fig:galewsky-high-res-b-16} show the solution at day 7 and 16, at which point spurious noise has polluted the conservative method, while the dissipative method appears to accurately represent the geostrophic turbulent flow and it's advection of $b$. 
\begin{figure}[!hbtp]
\begin{center}
	\includegraphics[width=0.8\textwidth]{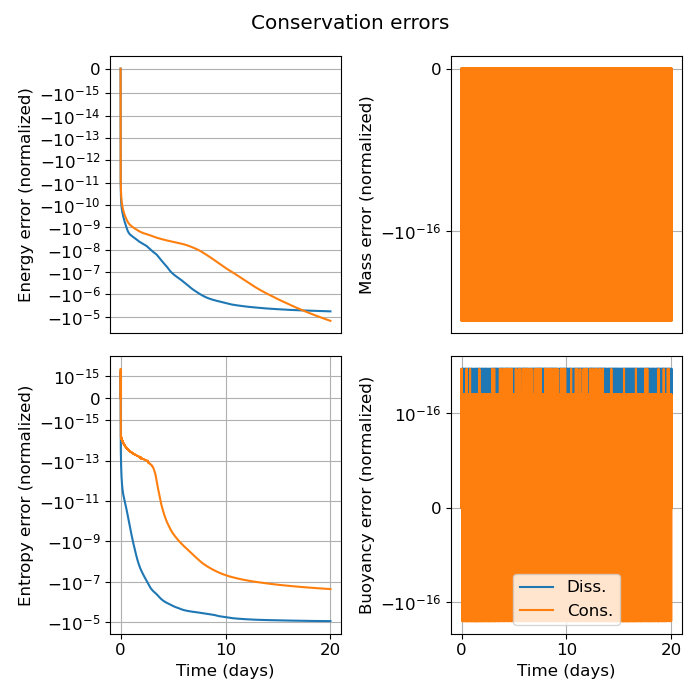}
 \caption{Conservation errors for modified Galewsky test case with the dissipative and conservative methods. A resolution of $6\times64\times64$ third order elements and a timestep of $0.8$ CFL was used.}
 \label{fig:conservation}
\end{center}
\end{figure}

\begin{figure}[!hbtp]
\begin{center}
 \includegraphics[width=0.48\textwidth]{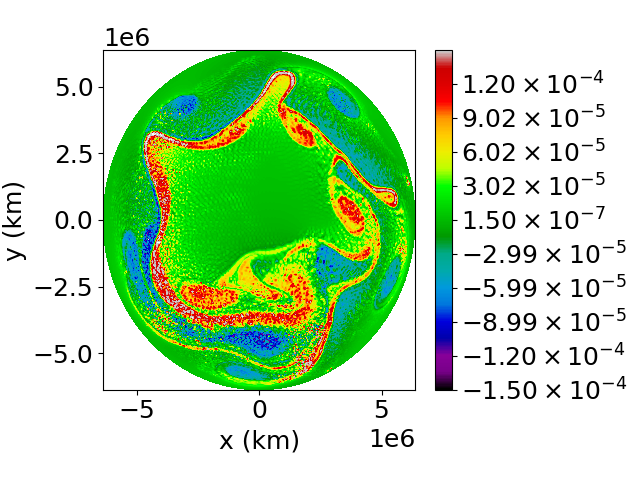}
 \includegraphics[width=0.48\textwidth]{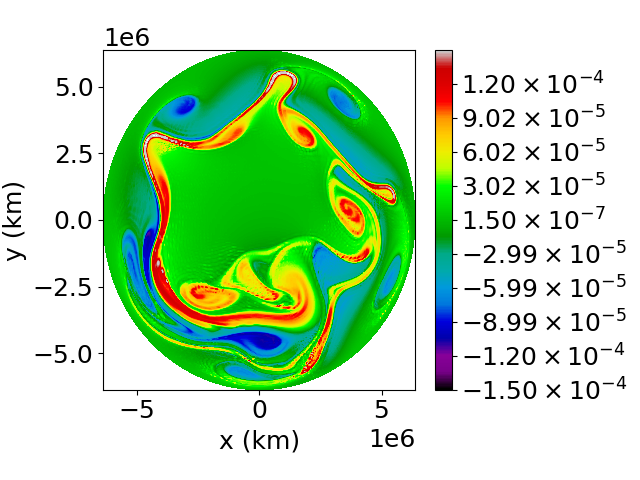} \\
 \caption{Relative vorticity for the modified Galewsky test case at day 7 for the conservative (left), and dissipative methods (right). A resolution of $6\times64\times64$ third order elements and a timestep of $0.8$ CFL was used.}
 \label{fig:galewsky-high-res-vort-7}
\end{center}
\end{figure}

\begin{figure}[!hbtp]
\begin{center}
 \includegraphics[width=0.48\textwidth]{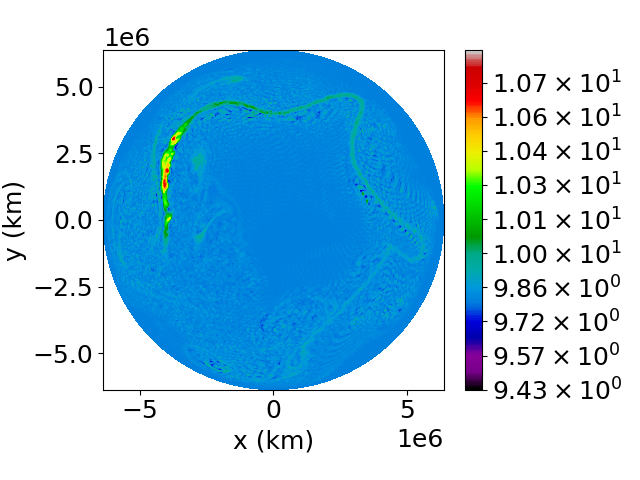}
 \includegraphics[width=0.48\textwidth]{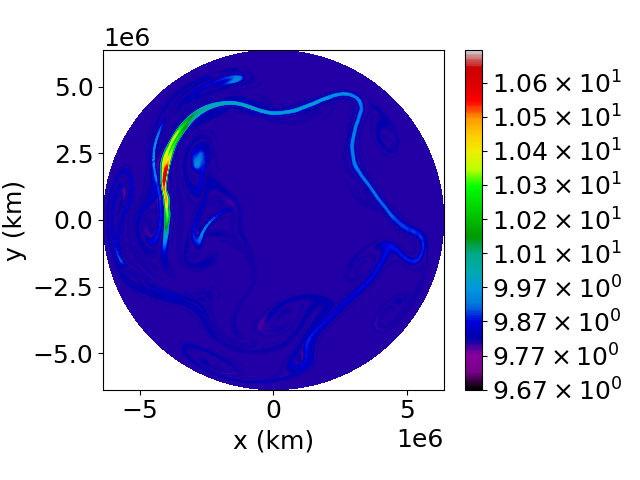} \\
 \caption{Buoyancy for the modified Galewsky test case at day 7 for the conservative (left), and dissipative methods (right). A resolution of $6\times64\times64$ third order elements and a timestep of $0.8$ CFL was used.}
 \label{fig:galewsky-high-res-b-7}
\end{center}
\end{figure}

\begin{figure}[!hbtp]
\begin{center}
 \includegraphics[width=0.48\textwidth]{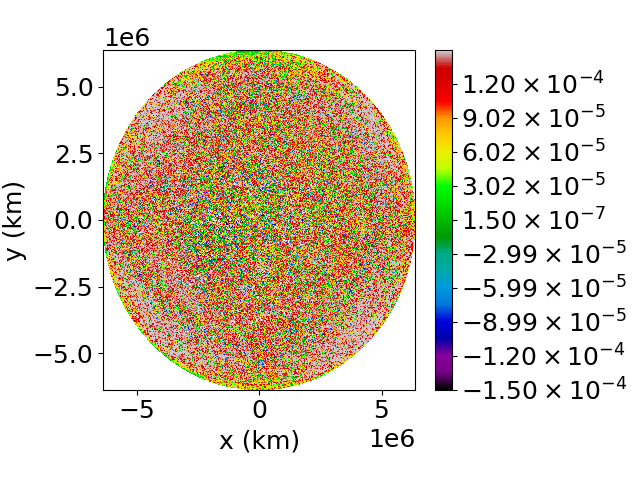}
 \includegraphics[width=0.48\textwidth]{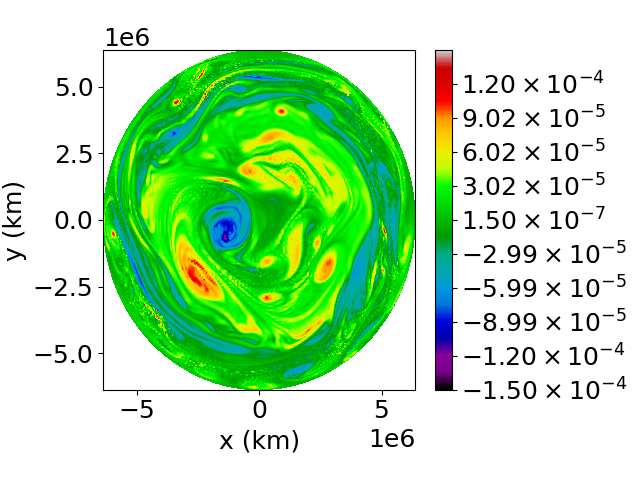} \\
 \caption{Relative vorticity for the modified Galewsky test case at day 16 for the conservative (left), and dissipative methods (right). A resolution of $6\times64\times64$ third order elements and a timestep of $0.8$ CFL was used.}
 \label{fig:galewsky-high-res-vort-16}
\end{center}
\end{figure}

\begin{figure}[!hbtp]
\begin{center}
 \includegraphics[width=0.48\textwidth]{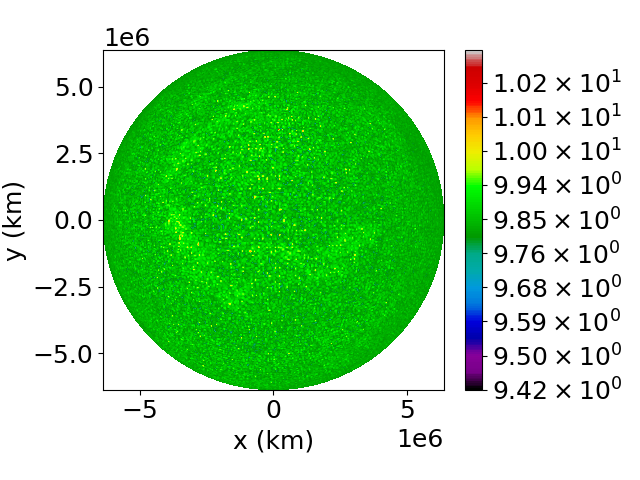}
 \includegraphics[width=0.48\textwidth]{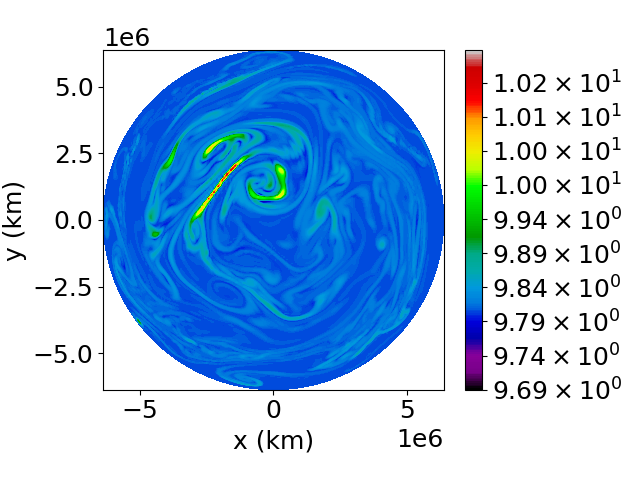} \\
 \caption{Buoyancy for the modified Galewsky test case at day 16 for the conservative (left), and dissipative methods (right). A resolution of $6\times64\times64$ third order elements and a timestep of $0.8$ CFL was used.}
 \label{fig:galewsky-high-res-b-16}
\end{center}
\end{figure}

\subsubsection{Stability}

To explore whether both entropy and energy stability are necessary for numerical stability we performed the low resolution experiment, $6 \times 16 \times 16$ 3rd order elements, with two variants of our discrete method. The first variant contains split terms in the buoyancy equation \eqref{eq:dg-b} but not the velocity equation \eqref{eq:dg-velocity}, and consequently conserves entropy but not energy. The second variant contains no split terms and only conserves energy. 

Figure \ref{fig:stability} shows the conservation errors. For this experiment entropy stability is both necessary and sufficient for numerical stability. The entropy of the energy conserving variant grows exponentially until the simulation becomes unstable at day 3. This numerical instability is preceded by a sharp uptick in energy. In contrast, the entropy of the entropy conserving monotonically decreases and the simulation remains stable for the 20 day period. Curiously, at the onset of numerical instability in the energy conserving variant, the entropy conserving variant also displays a sharp uptick in energy. This is followed by a gradual increase in energy until it stabilises around day 15. 

To understand why entropy stability alone may be sufficient for numerical stability we consider the energy growth of the semi-discrete entropy conserving variant, given by
\begin{equation}
    \mathcal{E}_t = \tfrac{1}{4}\left\langle \vb{F}, h\nabla b - \nabla (hb) - b\nabla h\right\rangle_\Omega\text{.}
\end{equation}
This error results from a failure to discretely mimic the chain rule, and the magnitude of this error is governed by the magnitude of the gradients $\nabla h$ and $\nabla b$. For geophysical fluid flows $h$ tends to be smooth, and entropy stability smooths out $\nabla b$. We hypothesize that this enables entropy stability to indirectly control the energy error and maintain numerical stability.

\begin{figure}[!hbtp]
\begin{center}
	\includegraphics[width=0.8\textwidth]{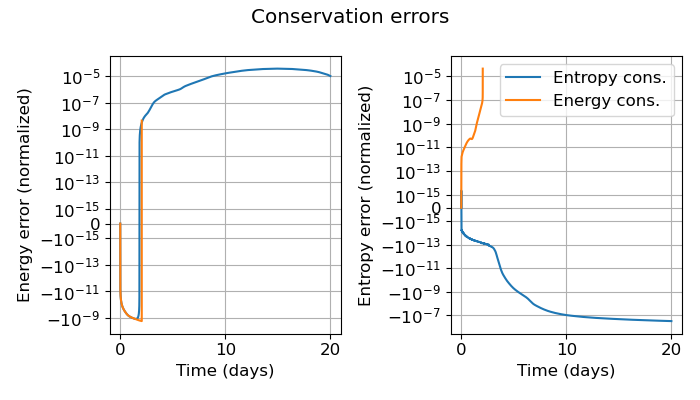}
 \caption{Conservation errors for modified Galewsky test case with the energy conserving and entropy conserving variants. A resolution of $6\times16\times16$ third order elements and a timestep of $0.8$ CFL was used. While the formulation that conserves energy but not entropy is stable for only a couple of days of simulation, the formulation that conserves entropy but not energy is stable for a full 20 days of simulation, long after the model has achieved a mature turbulent state.}
 \label{fig:stability}
\end{center}
\end{figure}

\FloatBarrier
\section{Conclusion}

In this paper, we extend the recent RSW DG method proposed in \cite{ricardo2023conservation} and present a novel DG method for the TRSW on a curvilinear mesh. To the best of our knowledge, this method represents the first approach for solving the TRSW on curvilinear geometry and the first DG method for the TRSW. Our main contributions lie in proving that our method semi-discretely conserves energy and buoyancy variance, and showing that this plays a central role in model stability. Notably, we establish that the variance of any advected quantity, such as the buoyancy in the TRSW, is a mathematical entropy function, and therefore our scheme is discretely entropy stable with an SSP time integrator. This approach can be trivially extended to stably advect any quantity. These theoretical findings are a direct outcome of our numerical fluxes and decomposition of the buoyancy advection and pressure terms. 

To validate our results, we conducted numerical experiments on a cubed sphere mesh, utilizing modified versions of the Galewksy \cite{galewsky2004initial} and Williamson 2 \cite{williamson1992standard} test cases. Through these simulations, we demonstrated the accuracy and stability of our method in effectively capturing geostrophic turbulence and the consequent turbulent advection of buoyancy. To the best of our knowledge, our approach represents the first in the literature capable of simulating well developed  geotrophic turbulunce on the sphere for the TRSW.

The TRSW were targeted in this paper due to their similarity with the compressible Euler equations, and future work will focus on extending this method to the Euler equations.

\section*{Declaration of competing interest}

The authors declare that they have no known competing financial interests or personal relationships that could have appeared to influence the work reported in this paper.

\section*{Data availability}

The code to generate, analyze, and plot the datasets used in this study can be found in the associated reproducibility repository https://github.com/kieranricardo/dg-TRSWe-code.

\clearpage
\bibliographystyle{unsrtnat}
\bibliography{main}  

\begin{thebibliography}{24}
\providecommand{\natexlab}[1]{#1}
\providecommand{\url}[1]{\texttt{#1}}
\expandafter\ifx\csname urlstyle\endcsname\relax
  \providecommand{\doi}[1]{doi: #1}\else
  \providecommand{\doi}{doi: \begingroup \urlstyle{rm}\Url}\fi

\bibitem[Staniforth and Thuburn(2012)]{staniforth2012horizontal}
A.~Staniforth and J.~Thuburn.
\newblock Horizontal grids for global weather and climate prediction models: a review.
\newblock \emph{Quarterly Journal of the Royal Meteorological Society}, 138\penalty0 (662):\penalty0 1--26, 2012.

\bibitem[Eldred et~al.(2019)Eldred, Dubos, and Kritsikis]{eldred2019quasi}
Christopher Eldred, Thomas Dubos, and Evaggelos Kritsikis.
\newblock A quasi-hamiltonian discretization of the thermal shallow water equations.
\newblock \emph{Journal of Computational Physics}, 379:\penalty0 1--31, 2019.

\bibitem[Kurganov et~al.(2020)Kurganov, Liu, and Zeitlin]{kurganov2020well}
Alexander Kurganov, Yongle Liu, and Vladimir Zeitlin.
\newblock A well-balanced central-upwind scheme for the thermal rotating shallow water equations.
\newblock \emph{Journal of Computational Physics}, 411:\penalty0 109414, 2020.

\bibitem[Kurganov et~al.(2021)Kurganov, Liu, and Zeitlin]{kurganov2021thermal}
Alexander Kurganov, Yongle Liu, and Vladimir Zeitlin.
\newblock Thermal versus isothermal rotating shallow water equations: comparison of dynamical processes by simulations with a novel well-balanced central-upwind scheme.
\newblock \emph{Geophysical \& Astrophysical Fluid Dynamics}, 115\penalty0 (2):\penalty0 125--154, 2021.

\bibitem[Waruszewski et~al.(2022)Waruszewski, Kozdon, Wilcox, Gibson, and Giraldo]{waruszewski2022entropy}
M.~Waruszewski, J.E. Kozdon, L.C. Wilcox, T.H. Gibson, and F.X. Giraldo.
\newblock Entropy stable discontinuous galerkin methods for balance laws in non-conservative form: Applications to the euler equations with gravity.
\newblock \emph{Journal of Computational Physics}, 468:\penalty0 111507, 2022.

\bibitem[Ranocha(2020)]{ranocha2020entropy}
Hendrik Ranocha.
\newblock Entropy conserving and kinetic energy preserving numerical methods for the euler equations using summation-by-parts operators.
\newblock \emph{Spectral and high order methods for partial differential equations ICOSAHOM 2018}, 134:\penalty0 525--535, 2020.

\bibitem[Ducros et~al.(2000)Ducros, Laporte, Soul{\`e}res, Guinot, Moinat, and Caruelle]{ducros2000high}
F~Ducros, F~Laporte, Th~Soul{\`e}res, Vincent Guinot, Ph~Moinat, and B~Caruelle.
\newblock High-order fluxes for conservative skew-symmetric-like schemes in structured meshes: application to compressible flows.
\newblock \emph{Journal of Computational Physics}, 161\penalty0 (1):\penalty0 114--139, 2000.

\bibitem[Morinishi(2010)]{morinishi2010skew}
Yohei Morinishi.
\newblock Skew-symmetric form of convective terms and fully conservative finite difference schemes for variable density low-mach number flows.
\newblock \emph{Journal of Computational Physics}, 229\penalty0 (2):\penalty0 276--300, 2010.

\bibitem[Sj{\"o}green and Yee(2019)]{sjogreen2019entropy}
Bj{\"o}rn Sj{\"o}green and HC~Yee.
\newblock Entropy stable method for the euler equations revisited: central differencing via entropy splitting and sbp.
\newblock \emph{Journal of Scientific Computing}, 81:\penalty0 1359--1385, 2019.

\bibitem[Gassner et~al.(2016{\natexlab{a}})Gassner, Winters, and Kopriva]{gassner2016split}
Gregor~J Gassner, Andrew~R Winters, and David~A Kopriva.
\newblock Split form nodal discontinuous galerkin schemes with summation-by-parts property for the compressible euler equations.
\newblock \emph{Journal of Computational Physics}, 327:\penalty0 39--66, 2016{\natexlab{a}}.

\bibitem[Hennemann et~al.(2021)Hennemann, Rueda-Ram{\'\i}rez, Hindenlang, and Gassner]{hennemann2021provably}
Sebastian Hennemann, Andr{\'e}s~M Rueda-Ram{\'\i}rez, Florian~J Hindenlang, and Gregor~J Gassner.
\newblock A provably entropy stable subcell shock capturing approach for high order split form dg for the compressible euler equations.
\newblock \emph{Journal of Computational Physics}, 426:\penalty0 109935, 2021.

\bibitem[Pirozzoli(2010)]{pirozzoli2010generalized}
Sergio Pirozzoli.
\newblock Generalized conservative approximations of split convective derivative operators.
\newblock \emph{Journal of Computational Physics}, 229\penalty0 (19):\penalty0 7180--7190, 2010.

\bibitem[Ricardo et~al.(2023{\natexlab{a}})Ricardo, Lee, and Duru]{ricardo2023conservation}
K.~Ricardo, D.~Lee, and K.~Duru.
\newblock Conservation and stability in a discontinuous galerkin method for the vector invariant spherical shallow water equations.
\newblock \emph{arXiv preprint arXiv:2303.17120}, 2023{\natexlab{a}}.

\bibitem[Gassner(2013)]{gassner2013skew}
Gregor~J Gassner.
\newblock A skew-symmetric discontinuous galerkin spectral element discretization and its relation to sbp-sat finite difference methods.
\newblock \emph{SIAM Journal on Scientific Computing}, 35\penalty0 (3):\penalty0 A1233--A1253, 2013.

\bibitem[Ricardo et~al.(2023{\natexlab{b}})Ricardo, Lee, and Duru]{ricardo2023entropy}
Kieran Ricardo, David Lee, and Kenneth Duru.
\newblock Entropy and energy conservation for thermal atmospheric dynamics using mixed compatible finite elements.
\newblock \emph{arXiv preprint arXiv:2305.12343}, 2023{\natexlab{b}}.

\bibitem[LeVeque(1992)]{leveque1992numerical}
R.J. LeVeque.
\newblock \emph{Numerical methods for conservation laws}, volume 214.
\newblock Springer, 1992.

\bibitem[Giraldo(2020)]{giraldo2020introduction}
Francis~X Giraldo.
\newblock \emph{An Introduction to Element-Based Galerkin Methods on Tensor-Product Bases: Analysis, Algorithms, and Applications}, volume~24.
\newblock Springer Nature, 2020.

\bibitem[Gassner et~al.(2016{\natexlab{b}})Gassner, Winters, and Kopriva]{gassner2016well}
Gregor~J Gassner, Andrew~R Winters, and David~A Kopriva.
\newblock A well balanced and entropy conservative discontinuous galerkin spectral element method for the shallow water equations.
\newblock \emph{Applied Mathematics and Computation}, 272:\penalty0 291--308, 2016{\natexlab{b}}.

\bibitem[Taylor and Fournier(2010)]{taylor2010compatible}
Mark~A Taylor and Aim{\'e} Fournier.
\newblock A compatible and conservative spectral element method on unstructured grids.
\newblock \emph{Journal of Computational Physics}, 229\penalty0 (17):\penalty0 5879--5895, 2010.

\bibitem[Kopriva and Gassner(2010)]{kopriva2010quadrature}
David~A Kopriva and Gregor Gassner.
\newblock On the quadrature and weak form choices in collocation type discontinuous galerkin spectral element methods.
\newblock \emph{Journal of Scientific Computing}, 44:\penalty0 136--155, 2010.

\bibitem[Shu and Osher(1988)]{shu1988efficient}
C.W. Shu and S.~Osher.
\newblock Efficient implementation of essentially non-oscillatory shock-capturing schemes.
\newblock \emph{Journal of computational physics}, 77\penalty0 (2):\penalty0 439--471, 1988.

\bibitem[Ran{\v{c}}i{\'c} et~al.(1996)Ran{\v{c}}i{\'c}, Purser, and Mesinger]{ranvcic1996global}
M.~Ran{\v{c}}i{\'c}, R.J. Purser, and F.~Mesinger.
\newblock A global shallow-water model using an expanded spherical cube: Gnomonic versus conformal coordinates.
\newblock \emph{Quarterly Journal of the Royal Meteorological Society}, 122\penalty0 (532):\penalty0 959--982, 1996.

\bibitem[Williamson et~al.(1992)Williamson, Drake, Hack, Jakob, and Swarztrauber]{williamson1992standard}
David~L Williamson, John~B Drake, James~J Hack, R{\"u}diger Jakob, and Paul~N Swarztrauber.
\newblock A standard test set for numerical approximations to the shallow water equations in spherical geometry.
\newblock \emph{Journal of computational physics}, 102\penalty0 (1):\penalty0 211--224, 1992.

\bibitem[Galewsky et~al.(2004)Galewsky, Scott, and Polvani]{galewsky2004initial}
Joseph Galewsky, Richard~K Scott, and Lorenzo~M Polvani.
\newblock An initial-value problem for testing numerical models of the global shallow-water equations.
\newblock \emph{Tellus A: Dynamic Meteorology and Oceanography}, 56\penalty0 (5):\penalty0 429--440, 2004.

\end{thebibliography}

\end{document}